\documentclass[11pt,a4paper,leqno]{amsart}
\usepackage[T1]{fontenc}
\usepackage[utf8]{inputenc}
\usepackage{lmodern}
\usepackage{microtype}

\usepackage{esint}
\usepackage{amsmath}
\usepackage{amsfonts}
\usepackage{amssymb}

\usepackage{geometry}

\usepackage[unicode, pdfusetitle]{hyperref}

\let\C\undefined

\usepackage[abbrev,backrefs]{amsrefs}
\usepackage{cleveref}

\numberwithin{equation}{section}

\newtheorem{proposition}{Proposition}
\newtheorem{theorem}[proposition]{Theorem}
\newtheorem{lemma}[proposition]{Lemma}
\theoremstyle{definition}

\theoremstyle{remark}

\newcommand{\defeq}{\coloneqq}

\newcommand{\Nset}{\mathbb{N}}
\newcommand{\Rset}{\mathbb{R}}

\newcommand{\BMO}{\mathrm{BMO}}

\newcommand{\Deriv}{D}

\newcommand{\dif}{\,\mathrm{d}}
\usepackage{mathtools}

\DeclarePairedDelimiter{\brk}{(}{)}
\DeclarePairedDelimiter{\abs}{\lvert}{\rvert}
\DeclarePairedDelimiter{\norm}{\lVert}{\rVert}

\DeclarePairedDelimiterX{\intvc}[2]{[}{]}{#1,#2}
\DeclarePairedDelimiterX{\intvl}[2]{(}{]}{#1,#2}
\DeclarePairedDelimiterX{\intvr}[2]{[}{)}{#1,#2}
\DeclarePairedDelimiterX{\intvo}[2]{(}{)}{#1,#2}

\providecommand{\st}{\,\vert\,}
\newcommand\stSymbol[1][]{%
\nonscript\;#1\vert
\allowbreak
\nonscript\;
\mathopen{}}
\DeclarePairedDelimiterX\set[1]\{\}{%
\renewcommand\st{\stSymbol[\delimsize]}
#1
}

\DeclareMathOperator{\supp}{supp}
\DeclareMathOperator{\diam}{diam}

\newcommand{\maxfun}{\mathcal{M}}

\usepackage{constants}

\author{Jean Van Schaftingen}

\title[Fractional Gagliardo--Nirenberg and BMO]{
Fractional Gagliardo-Nirenberg interpolation inequality and bounded mean oscillation
}

\date{January 17, 2023}

\address{
Universit\'e catholique de Louvain, Institut de Recherche en Math\'ematique et Physique, Chemin du Cyclotron 2 bte L7.01.01, 1348 Louvain-la-Neuve, Belgium}
\email{Jean.VanSchaftingen@UCLouvain.be}



\begin{document}

\begin{abstract}
We prove Gagliardo--Nirenberg interpolation inequalities estimating the Sobolev semi-norm in terms of the bounded mean oscillation semi-norm and of a Sobolev semi-norm, with some of the Sobolev semi-norms having fractional order.
\end{abstract}

\thanks{The author was supported by the
Projet de Recherche T.0229.21 ``Singular Harmonic Maps and Asymptotics of Ginzburg--Landau Relaxations'' of the
Fonds de la Recherche Scientifique--FNRS}

\maketitle 

\section{Introduction}

The homogeneous Gagliardo--Nirenberg interpolation inequality for Sobolev space states that if \(d \in \Nset \setminus \set{0}\) and if \(0 \le s_0 < s < s_1\), \(1 \le p, p_0, p_1 \le \infty\) and \(0 < \theta < 1\) fulfil the condition 
\begin{equation}
  (s, \tfrac{1}{p}) = (1 - \theta) \brk[\big]{s_0, \tfrac{1}{p_0}} + \theta\brk[\big]{s_1, \tfrac{1}{p_1}},
\end{equation} 
then, for every function \(f \in \dot{W}^{s_0, p_0} (\Rset^d) \cap \dot{W}^{s_1, p_1} (\Rset^d)\), one has \(f \in \dot{W}^{s, p} (\Rset^d)\), and
\begin{equation}
\label{eq_eev7phaipouxooGuoYooGhah}
  \norm{f}_{\dot{W}^{s, p}(\Rset^d)} \le C \norm{f}_{\dot{W}^{s_0, p_0}(\Rset^d)}^{1 - \theta}  \norm{f}_{\dot{W}^{s_1, p_1}(\Rset^d)}^{\theta},
\end{equation}
unless \(s_1\) is an integer, \(p_1 = 1\) and \(s_1 - s_0 \le 1 - \frac{1}{p_0}\).

When \(s = 0\), we use the convention that \(\dot{W}^{0, p} (\Rset^d) = L^p (\Rset^d)\), and when \(s \in\Nset \setminus \set{0}\) is a positive integer, \(\dot{W}^{s, p} (\Rset^d)\) is the classical integer-order \emph{homogeneous Sobolev space} of \(s\) times weakly differentiable functions \(f : \Rset^d \to \Rset\) such that \(\Deriv^s f \in L^{p} (\Rset^d)\) and
\begin{equation}
\norm{f}_{\dot{W}^{s, p} (\Rset^d)} 
\defeq \brk*{\int_{\Rset^d} \abs{\Deriv^s f}^p}^\frac{1}{p}.
\end{equation}
For \(s_0, s_1, s \in \Nset\) the inequality \eqref{eq_eev7phaipouxooGuoYooGhah} was proved by Gagliardo \cite{Gagliardo_1959} and Nirenberg \cite{Nirenberg_1959} (see also \cite{Fiorenza_Formica_Roskovec_Soudsky_2021}).

When \(s \not \in \Nset\), the \emph{homogeneous fractional Sobolev--Slobodeckiĭ space} \(\dot{W}^{s, p} (\Rset^d)\) can be defined as the set of measurable functions \(f : \Rset^d \to \Rset\) which are \(k\) times weakly differentiable with a finite Gagliardo semi-norm:
\begin{equation}
  \norm{f}_{\dot{W}^{s, p} (\Rset^d)} 
  \defeq \brk*{\int_{\Rset^d} \int_{\Rset^d} \frac{\abs{\Deriv^{k} f (y) - \Deriv^{k} f (x)}^p}{\abs{y - x}^{d + \sigma p}} \dif y \dif x}^\frac{1}{p} < \infty,
\end{equation}
with \(k \in \Nset\), \(\sigma \in \intvo{0}{1}\) and \(s = k + \sigma\);
the characterisation of the range in which the Gagliardo--Nirenberg interpolation inequality \eqref{eq_eev7phaipouxooGuoYooGhah} holds was performed in a series of works \citelist{\cite{Cohen_2000}\cite{Brezis_Mironescu_2001}\cite{Cohen_Dahmen_Daubechies_DeVore_2003}\cite{Cohen_Meyer_Oru_1998}} up to the final complete settlement by Brezis and Mironescu \cite{Brezis_Mironescu_2018}.

We focus on the endpoint case where \(s_0 = 0\) and \(p_0 = \infty\). 
In this case, the inequality \eqref{eq_eev7phaipouxooGuoYooGhah} becomes 
\begin{equation}
\label{eq_ouzieFa8NeeSoonei0pae8ku}
  \norm{f}_{\dot{W}^{s, p}(\Rset^d)}^p \le C \norm{f}_{L^\infty (\Rset^d)}^{p - p_1}  \norm{f}_{\dot{W}^{s_1, p_1}(\Rset^d)}^{p_1},
\end{equation}
and holds under the assumption that \(s p = s_1 p_1\) and either \(s_1 \ne 1\) or \(p_1 > 1\).
It is natural to ask whether the inequality \eqref{eq_ouzieFa8NeeSoonei0pae8ku} can be strengthened by replacing the uniform norm \(\norm{\cdot}_{L^\infty (\Rset^d)}\) by John and Nirenberg’s \emph{bounded mean oscillation} (BMO) semi-norm \(\norm{\cdot}_{\BMO(\Rset^d)}\), which plays an important role in harmonic analysis, calculus of variations and partial differential equations \cite{John_Nirenberg_1961},
that is, whether we have the inequality
\begin{equation}
\label{eq_iethieNie4Zou8ohch0cee2n}
  \norm{f}_{\dot{W}^{s, p}(\Rset^d)}^p \le 
  C \norm{f}_{\BMO(\Rset^d)}^{p - p_1}  \norm{f}_{\dot{W}^{s_1, p_1}(\Rset^d)}^{p_1},
\end{equation}
where the bounded mean oscillation semi-norm \(\norm{\cdot}_{\BMO(\Rset^d)}\) is defined for any measurable function \(f : \Rset^d \to \Rset\) as
\begin{equation}
\label{eq_oothaKahp8pe7OGheix4phoo}
 \norm{f}_{\BMO(\Rset^d)}
 \defeq 
 \sup_{\substack {x \in \Rset^d\\ r > 0}}
 \fint_{B_r (x)} \fint_{B_r (x)} \abs{f (y) - f (z)}\dif y \dif z.
\end{equation}
The estimate \eqref{eq_iethieNie4Zou8ohch0cee2n} was proved indeed when \(s = 1\), \(p = 4\), \(s_1 = 2\) and \(p_1 = 2\) via a Littlewood--Paley decomposition by Meyer and Rivière \citelist{\cite{Meyer_Riviere_2003}*{theorem 1.4}},
and for  \(s, s_1 \in \Nset\) via the duality between \(\mathrm{BMO} (\Rset^d)\) and the real Hardy space \(\mathcal{H}^1 (\Rset^d)\) by Strezelecki \cite{Strzelecki_2006}; a direct proof was been given recently by Miyazaki \cite{Miyazaki_2020} (in the limiting case \(s_0 = s_1 = 0\), see \citelist{\cite{Kozono_Wadade_2008}*{theorem 2.2}\cite{Chen_Zhu_2005}});
when \(s_1 < 1\), the estimate \eqref{eq_iethieNie4Zou8ohch0cee2n} has been proved by Brezis and Mironescu through a Littlewood--Paley decomposition \citelist{\cite{Brezis_Mironescu_2021}*{lemma 15.7}} (see also \citelist{\cite{Adams_Frazier_1992}\cite{Kozono_Wadade_2008}} for similar estimates in Riesz potential spaces).

The main result (\cref{proposition_GN_Wsp_W1p_BMO}) of the present work is the estimate \eqref{eq_iethieNie4Zou8ohch0cee2n} when \(s_1 = 1\) and \(0 < s < 1\), with a proof which is quite elementary: the main analytical tool is the classical maximal function theorem. 
We also show how the same ideas can be used to give a direct proof of \eqref{eq_iethieNie4Zou8ohch0cee2n} when \(s_1 < 1\), depending only on the definitions of the Gagliardo and bounded mean oscillation semi-norms (\cref{proposition_GN_Wsp_Wsp_BMO}).
Finally, we show how a last interpolation result (\cref{proposition_higher_order}) allows one to obtain the full range of interpolation between \(\mathrm{BMO} (\Rset^d)\) and higher-order fractional Sobolev--Slobodeckiĭ spaces \(\dot{W}^{s, p} (\Rset^d)\) with \(s \in \intvo{1}{\infty}\).

Our proofs can be considered as fractional counterparts of Miyazaki's direct proof in the integer-order case \cite{Miyazaki_2020}. We also refer to Dao's recent work \cite{Dao} for an alternative approach via negative-order Besov spaces to the results in the present paper.

\section{Interpolation between first-order Sobolev semi-norm and mean oscillation}

We prove the following interpolation inequality between the fist-order Sobolev semi-norm and the mean oscillation seminorm into fractional Sobolev spaces.

\begin{theorem}
  \label{proposition_GN_Wsp_W1p_BMO}
For every \(d \in \Nset \setminus \set{0}\) and every \(p\in \intvo{1}{\infty}\), there exists a constant \(C(p) > 0\) such that for every \(s \in \intvo{1/p}{1}\), every open convex set \(\Omega \subseteq \Rset^d\) satisfying \(\varkappa (\Omega) < \infty\) and every function \(f \in \dot{W}^{1, sp} (\Omega) \cap \BMO (\Omega)\), one has
\(f \in \dot{W}^{s, p} (\Omega)\) and 
  \begin{equation}
  \label{eq_aijei2Xaighei4obeichiiKi}
  \iint\limits_{\Omega \times \Omega}
  \frac{\abs{f (y) - f (x)}^p}{\abs{y - x}^{d + sp}}
  \dif y
  \dif x \le 
  \frac{C(p) \varkappa(\Omega)^{sp}}{(sp - 1)(1 - s)} 
  \norm{f}_{\mathrm{BMO}(\Omega)}^{(1 -s) p}
  \int_{\Omega} \abs{\Deriv f}^{sp}.
  \end{equation}
\end{theorem}

We define here for a domain \(\Omega \subseteq \Rset^d\), the \emph{bounded mean oscillation semi-norm} of a measurable function \(f: \Omega \to \Rset\) as
\begin{equation}
\label{eq_HakeeGhae0ku0yohy1jied9E}
 \norm{f}_{\BMO(\Omega)}
 \defeq 
 \sup_{\substack {x \in \Omega\\ r > 0}}
 \fint_{\Omega \cap B_r (x)} \fint_{\Omega \cap B_r (x)} \abs{f (y) - f (z)}\dif y \dif z,
\end{equation}
and the geometric quantity
\begin{equation}
\label{eq_waiC3theiwooDaithee1aeLa}
 \varkappa (\Omega)
 \defeq \sup \set*{\frac{\mathcal{L}^d (B_r (x))}{\mathcal{L}^d (\Omega \cap B_r (x))}
 \st x \in \Omega \text{ and } r \in \intvo{0}{\diam (\Omega)} }.
\end{equation}
For the latter quantity, one has for example
\begin{equation}
 \varkappa (\Rset^d) = 1
\end{equation}
and 
\begin{equation}
 \varkappa (\Rset^d_+) = 2.
\end{equation}
If the set \(\Omega\) is convex and bounded, we have \(\Omega \subseteq B_{\diam (\Omega)} (x)\) and 
\(t \Omega + (1 - t)x \subseteq \Omega \cap B_r (x)\), with \(t \defeq r/\diam (\Omega)\), so that 
\begin{equation*}
\mathcal{L}^d (\Omega \cap B_{r} (x)) 
\ge 
t^d \mathcal{L}^d (\Omega)
= 
 \frac{\mathcal{L}^d (\Omega)}{\diam (\Omega)^d}r^d,
\end{equation*}
and thus 
\begin{equation}
 \varkappa (\Omega) \le \frac{\mathcal{L}^d (B_1)} {\mathcal{L}^d (\Omega)} \diam (\Omega)^d.
\end{equation}
The quantity \(\varkappa (\Omega)\) can be infinite for some unbounded convex sets such as  \(\Omega = \intvo{0}{1}\times \Rset^{d - 1}\) and \(\Omega = \set{(x', x_d) \in \Rset^d \st x_d \ge \abs{x'}^2}\).

Our first tool to prove \cref{proposition_GN_Wsp_W1p_BMO} is an estimate by the maximal function of the derivative of the average distance of values on a ball to a fixed value; 
this formula is related to the \emph{Lusin--Lipschitz inequality} \citelist{\cite{Liu_1977}*{lemma\ 2}\cite{Acerbi_Fusco_1984}*{lemma\ II.1}\cite{Bojarski_1990}\cite{Hajlasz_1996}*{p.\ 404}\cite{Jabin_2010}*{(3.3)}}.

\begin{lemma}
\label{lemma_osc_MDu}
If the set \(\Omega \subseteq \Rset^d\) is open and convex and if \(f \in \dot{W}^{1, 1}_{\mathrm{loc}} (\Omega)\), then for every \(r \in \intvo{0}{\diam(\Omega)}\) and almost every \(x \in \Omega\), 
\begin{equation}
\label{eq_eutai3kaeviuGhaisiechoqu}
\fint_{\Omega \cap B_{r} (x)} \abs{f (z) - f (x)} \dif z
\le  \varkappa (\Omega)\, r\,
\maxfun \abs{\Deriv f} (x).
\end{equation}
\end{lemma}

Here \(\maxfun g : \Rset^d \to \intvc{0}{+\infty}\) denotes the classical \emph{Hardy--Littlewood maximal function} of the function \(g : \Omega \to \Rset\), defined for each \(x \in \Rset^d\) by 
\begin{equation}
\label{eq_JahvaePieG5cahth4AepheCi}
\maxfun g (x)
\defeq
\sup_{r > 0} \frac{1}{\mathcal{L}^d (B_r (x))} \int_{\Omega \cap B_r (x)} \abs{g}.
\end{equation}

\begin{proof}[Proof of \cref{lemma_osc_MDu}]
Since \(\Omega\) is convex and \(f \in \dot{W}^{1, 1} (\Omega)\), for almost every \(x \in \Omega\) and every \(r \in \intvo{0}{\infty}\), we have 
\begin{equation}
\label{eq_wa1lohNgiechei5oos0ae0Oh}
 \int_{\Omega \cap B_{r} (x)} \abs[\big]{f (z) - f (x)} \dif z
\le {\int_{\Omega \cap B_{r}(x)}} \int_{0}^{1} \abs{\Deriv f ((1-t) x + t z)[z - x]} \dif t \dif z.
\end{equation}
By convexity of the set \(\Omega\), for every \(z \in \Omega \cap B_{r}(x)\) and \(t \in \intvc{0}{1}\) we have \((1 -t)x + t z \in \Omega \cap B_{t r}(x)\).
We deduce from \eqref{eq_JahvaePieG5cahth4AepheCi} and \eqref{eq_wa1lohNgiechei5oos0ae0Oh}
through the change of variable \(y = (1-t)x + tz\) that 
\begin{equation}
\begin{split}
\int_{\Omega \cap B_{r} (x)} \abs{f (z) - f (x)} \dif z
&\le \int_0^1 \int_{\Omega \cap B_{tr} (x)} \frac{\abs{\Deriv f (y)[y - x]}}{t^{d + 1}} \dif y \dif t\\
&\le r \maxfun{\abs{\Deriv f}}(x) \int_0^1 \frac{\mathcal{L}^d (B_{tr} (x))}{t^{d}}  \dif t\\
 &\le r \mathcal{L}^d (B_r (0))
 \maxfun \abs{\Deriv f} (x),
\end{split}
\end{equation}
in view of the definition \eqref{eq_HakeeGhae0ku0yohy1jied9E} of the maximal function, 
and the conclusion \eqref{eq_eutai3kaeviuGhaisiechoqu} then follows from the definition of the geometric quantity \(\varkappa (\Omega)\) in \eqref{eq_waiC3theiwooDaithee1aeLa}.
\end{proof}

Our second tool to prove \cref{proposition_GN_Wsp_W1p_BMO} is the following property of averages of functions of bounded mean oscillation (see \citelist{\cite{Carleson_1976}*{\S 3}}).

\begin{lemma}
\label{lemma_BMO_log}
If the set \(\Omega \subseteq \Rset^d\) is open and convex, if \(f \in \BMO (\Omega) \)
and if \(r_0 < r_1\),  then 
\begin{equation}
\label{eq_eiboh7im4ZahghuLaiv9Eene}
  \fint_{\Omega \cap B_{r_0}(x)} \fint_{\Omega \cap B_{r_1} (x)}
  \abs{f (y) - f (z)}\dif y \dif z
  \le  e \brk[\big]{1 + d \ln \brk*{r_1/r_0}} \norm{f}_{\BMO(\Omega)}.
\end{equation}
\end{lemma}

In \eqref{eq_eiboh7im4ZahghuLaiv9Eene}, \(e\) denotes Euler's number.

The proof of \cref{lemma_BMO_log} will use the following triangle inequality for averages
\begin{lemma}
\label{lemma_triangle}
Let \(\Omega \subseteq \Rset^d\). If the function \(f : \Omega \to \Rset\) is measurable, and the sets \(A, B, C \subseteq \Rset^d\) are measurable and have positive measure, then 
\begin{equation*}
 \fint_{A} \fint_{B} \abs{f (y) - f (x)} \dif y \dif x
 \le \fint_{A} \fint_{C} \abs{f (z) - f (x)} \dif z \dif x
 +
 \fint_{C} \fint_{B} \abs{f (y) - f (z)} \dif y \dif z.
\end{equation*}
\end{lemma}
\begin{proof}
We have successively, in view of the triangle inequality,
\begin{equation}
\begin{split}
 \fint_{A} \fint_{B} \abs{f (y) - f (x)} \dif y \dif x
 &= \fint_{A} \fint_{B} \fint_{C}  \abs{f (y) - f (x)} \dif z \dif y \dif x\\
 & \le \fint_{A} \fint_{B}  \abs{f (z) - f (x)} + \abs{f (y) - f (z)} \dif z \dif y \dif x\\
 &= \fint_{A} \fint_{C} \abs{f (z) - f (x)} \dif z \dif x
 +
 \fint_{C} \fint_{B} \abs{f (y) - f (z)} \dif y \dif z. \qedhere
 \end{split}
\end{equation}

\end{proof}

\begin{proof}[Proof of \cref{lemma_BMO_log}]
We first note that since \(r_1 > r_0\), we have in view of \eqref{eq_HakeeGhae0ku0yohy1jied9E}
\begin{equation}
\label{eq_voh1xobai3Eshai7aeh8Na1u}
\begin{split}
  \fint_{\Omega \cap B_{r_0}(x)} &\fint_{\Omega \cap B_{r_1} (x)} \abs{f (y) - f (z)}\dif y \dif z\\
  &\le \frac{\mathcal{L}^d (\Omega \cap B_{r_1} (x))}{\mathcal{L}^d (\Omega \cap B_{r_0} (x))} \fint_{\Omega \cap B_{r_1}(x)} \fint_{\Omega \cap B_{r_1} (x)} \abs{f (y) - f (z)}\dif y \dif z\\
  &\le \brk*{\frac{r_1}{r_0}}^d \norm{f}_{\BMO(\Omega)},
\end{split}
\end{equation}
since by convexity \(r_0/r_1(\Omega \cap B_{r_1} (x)) \subseteq\Omega \cap B_{r_0} (x)\) and thus
\(\mathcal{L}^d (\Omega \cap B_{r_1} (x))/r_1^d \le \mathcal{L}^d (\Omega \cap B_{r_0} (x))/r_0^d\).
Applying \(k \in \Nset\setminus \set{0}\) times the inequality \eqref{eq_voh1xobai3Eshai7aeh8Na1u}, we get thanks to the triangle inequality for mean oscillation of \cref{lemma_triangle},
\begin{equation}
\begin{split}
 \fint_{\Omega \cap B_{r_0}(x)}& \fint_{\Omega \cap B_{r_1} (x)} \abs{f (y) - f (z)}\dif y \dif z\\
&\le \sum_{j = 0}^{k - 1}\fint_{\Omega \cap B_{r_0(r_1/r_0)^{j/k}}(x)} \fint_{\Omega \cap B_{r_0(r_1/r_0)^{(j + 1)/k}} (x)} \abs{f (y) - f (z)}\dif y \dif z\\
 &\le k \brk*{\frac{r_1}{r_0}}^{d/k} \norm{f}_{\BMO(\Omega)}.
\end{split}
\end{equation}
Taking \(k \in \Nset\setminus \set{0}\) such that \(k - 1 < d \ln (r_1/r_0) \le k\), we obtain
the conclusion \eqref{eq_eiboh7im4ZahghuLaiv9Eene}.
\end{proof}

Our last tool to prove \cref{proposition_GN_Wsp_W1p_BMO} is the following integral identity.

\begin{lemma}
\label{lemma_Gamma}
For every \(p \in \intvo{1}{\infty}\) and \(\alpha \in \intvo{0}{\infty}\), one has 
\begin{equation*}
 \int_1^\infty \frac{(\ln r)^{p}}{r^{1 + \alpha}} \dif r = \frac{\Gamma (p + 1)}{\alpha^{p + 1}}.
\end{equation*}
\end{lemma}
\begin{proof}
One performs the change of variable \(r = \exp (t/\alpha)\) in the left-hand side integral and uses the classical integral definition of the Gamma function.
\end{proof}

We now proceed to the proof of \cref{proposition_GN_Wsp_W1p_BMO}.

\begin{proof}%
[Proof of \cref{proposition_GN_Wsp_W1p_BMO}]%
\resetconstant%
For every \(x, y \in \Omega\), we have by the triangle inequality and the domain monotonicity of the integral
  \begin{equation}
      \label{eq_ieph0eiwoose5IeJooxee1wj}
\begin{split}
    \abs{f (y) - f (x)}
    &\le 
    {\fint_{\Omega \cap B_{\abs{y - x}/2}(\frac{x + y}{2})}} \abs{f (y) - f}
    +{
    \fint_{\Omega \cap B_{\abs{y - x}/2}(\frac{x + y}{2})}} \abs{f - f (x)}\\
    & \le 
    {2^d \fint_{\Omega \cap B_{\abs{y - x}}(y)}} \abs{f (y) - f }
    +{2^d 
    \fint_{\Omega \cap B_{\abs{y - x}}(x)}} \abs{f - f (x)},
\end{split}
\end{equation}
since by convexity \(\Omega \cap B_{\abs{y - x}/2} (\frac{x + y}{2}) \subseteq \frac{1}{2} (B_{\abs{y - x}} (x) \cap \Omega) + \frac{y}{2}\).
It follows thus from \eqref{eq_ieph0eiwoose5IeJooxee1wj} by integration and by symmetry that 
\begin{equation}
\begin{split}
 \iint\limits_{\Omega \times \Omega}
 \frac{   \abs{f (y) - f (x)}^p}{\abs{y - x}^{d + sp}} \dif y \dif x
 & \le \C \iint\limits_{\Omega\times \Omega}
 \brk*{\fint_{\Omega \cap B_{\abs{y - x}}(x)} \abs{f - f (x)}}^p \frac{\dif y \dif x}{\abs{y - x}^{d + sp}}\\
 & \le \C \int_{\Omega} \int_0^{\diam \Omega}
 \brk*{\fint_{\Omega \cap B_r(x)} \abs{f - f (x)} }^p \frac{\dif r}{r^{1 + sp}}  \dif x.
\end{split}
\end{equation}
If \(\varrho \in \intvo{0}{\diam (\Omega)}\), we first have by \cref{lemma_osc_MDu}, for almost every \(x \in \Omega\),
\begin{equation}
\label{eq_pheiph4dei2uc0iefo5KaR0j}
\begin{split}
 \int_0^\varrho \brk*{\fint_{\Omega \cap B_r(x)} \abs{f - f (x)} }^p \frac{\dif r}{r^{1 + sp}}
 & \le \brk[\big]{\varkappa (\Omega) \, \maxfun \abs{\Deriv f} (x)}^p \int_0^{\varrho} r^{(1 - s)p - 1} \dif r \\
 &=  \frac{\varrho^{(1 - s)p} \brk[\big]{\varkappa (\Omega)\, \maxfun\abs{\Deriv f}(x)}^p}{(1 - s)p}.
 \end{split}
\end{equation}
Next we have by the triangle inequality, by \cref{lemma_osc_MDu} again and by \cref{lemma_BMO_log}, for every \(r \in \intvo{\varrho}{\diam (\Omega)}\),
\begin{equation}
\label{eq_AeJu8Uveed8teiqu1soR5que}
\begin{split}
\fint_{\Omega \cap B_r(x)} \abs{f - f (x)} & \le \fint_{\Omega \cap B_\varrho (x)} \abs{f - f (x)}
+ \fint_{\Omega \cap B_r (x)} \fint_{\Omega \cap B_\varrho (x)} \abs{f (y) - f (z)} \dif y \dif z\\
&\le  \brk*{\varrho\, \varkappa (\Omega)\,\maxfun{\abs{\Deriv f}}(x) + e (1 + d \ln (r/\varrho)) \norm{f}_{\BMO(\Omega)} },
\end{split}
\end{equation}
and hence, integrating \eqref{eq_AeJu8Uveed8teiqu1soR5que}, we get
\begin{equation}
\label{eq_Tah3Loosiela3aQuoothio2p}
\begin{split}
 \int_\varrho^{\diam (\Omega)} &\brk*{\fint_{\Omega \cap B_r(x)} \abs{f - f (x)}}^p \frac{\dif r}{r^{1 + sp}}\\
 & \le \C \biggl(\int_{\varrho}^{\infty} \frac{\varrho^p \maxfun{\abs{\Deriv f}}(x)^p}{r^{1 + sp}} \dif r \\
 &\hspace{4cm}
 + \int_{\varrho}^\infty  \varkappa (\Omega)^p \norm{f}_{\BMO(\Omega)}^p \frac{(1 + d \ln(r/\varrho))^p}{r^{1 + sp}} \dif r\biggr)\\
 &\le \C \brk*{\frac{\varrho^{(1 - s)p} \varkappa(\Omega)^p \maxfun\abs{\Deriv f}(x)^p}{s} + \frac{\Gamma (p + 1) \norm{f}_{\BMO(\Omega)}^p}{(sp)^{p + 1}\varrho^{sp}}},
 \end{split}
\end{equation}
in view of \cref{lemma_Gamma}.
Putting \eqref{eq_pheiph4dei2uc0iefo5KaR0j} and \eqref{eq_Tah3Loosiela3aQuoothio2p} together, we get,  since \(sp > 1\),
\begin{equation}
\label{eq_quohjeicaedaiY4aewoh5Ooc}
\begin{split}
 \int_0^{\diam (\Omega)} \brk*{\fint_{\Omega \cap B_r(x)} \abs[\big]{f - f (x)} }^p &\frac{\dif r}{r^{1 + sp}}\\
 & \le \C\brk*{\frac{\varrho^{(1 - s) p}  \varkappa (\Omega)^p \maxfun\abs{\Deriv f}(x)^p}{1 - s} + \frac{\norm{f}_{\BMO(\Omega)}^p}{\varrho^{sp}}}.
\end{split}
\end{equation}
If \(\norm{f}_{\mathrm{BMO}(\Omega)} \le \diam (\Omega) \varkappa (\Omega) \maxfun \abs{\Deriv f}(x)\), taking \(\varrho \defeq \norm{f}_{\mathrm{BMO}(\Omega)}/(\varkappa (\Omega) \maxfun \abs{\Deriv f}(x))\) in \eqref{eq_quohjeicaedaiY4aewoh5Ooc}, we obtain 
\begin{equation}
\label{eq_oophae8ocee6Ohphoh4foox8}
  \int_0^{\diam (\Omega)} \brk*{\fint_{\Omega \cap B_r(x)} \abs[\big]{f - f (x)}}^p \frac{\dif r}{r^{1 + sp}} \le  \frac{\C}{1 - s}  \brk*{\varkappa (\Omega)\maxfun\abs{\Deriv f}(x)}^{sp} \,\norm{f}_{\BMO(\Omega)}^{(1-s)p};
\end{equation}
otherwise we take \(\varrho \defeq \diam (\Omega) \le \norm{f}_{\mathrm{BMO}(\Omega)}/(\varkappa (\Omega) \maxfun \abs{\Deriv f}(x))\) in \eqref{eq_pheiph4dei2uc0iefo5KaR0j} and also obtain \eqref{eq_oophae8ocee6Ohphoh4foox8}.
Integrating the inequality \eqref{eq_oophae8ocee6Ohphoh4foox8}, we reach the conclusion \eqref{eq_aijei2Xaighei4obeichiiKi} by the quantitative version of the classical maximal function theorem in \(L^{sp}(\Rset^d)\)  since \(sp > 1\) (see for example \citelist{\cite{Stein_1970}*{theorem I.1}}).
\end{proof}

We conclude this section by pointing out that  \cref{proposition_GN_Wsp_W1p_BMO} admits a \emph{localised version} in terms of Fefferman and Stein's 
 \emph{sharp maximal function} \(f^\sharp:\Omega \to \intvc{0}{\infty}\) which is defined for every \(x \in \Omega\) (see \citelist{\cite{Fefferman_Stein_1972}*{(4.1)}}) as
\begin{equation}
 f^\sharp (x) \defeq \sup_{r > 0} \fint_{\Omega \cap B_r (x)} \fint_{\Omega \cap B_r (x)} \abs{f (y) - f (z)}\dif y \dif z;
\end{equation}
noting that the proof of \cref{lemma_BMO_log} yields in fact the estimate
\begin{equation}
  \fint_{\Omega \cap B_{r_0}(x)} \fint_{\Omega \cap B_{r_1} (x)}
  \abs{f (y) - f (z)}\dif y \dif z
  \le  e \brk[\big]{1 + d \ln \brk*{r_1/r_0}} f^\sharp (x)
\end{equation}
and following then the proof of \cref{proposition_GN_Wsp_W1p_BMO}, we reach the following local counterpart of \eqref{eq_oophae8ocee6Ohphoh4foox8}.

\begin{proposition}
\label{proposition_local_GN_Wsp_W1p_BMO}
For every \(d \in \Nset \setminus \set{0}\) and for every \(p\in \intvo{1}{\infty}\), there exists a constant \(C > 0\) such that for every \(s \in \intvo{1/p}{1}\), for every open convex set \(\Omega \subseteq \Rset^d\) satisfying \(\varkappa (\Omega) < \infty\), for every function \(f \in \dot{W}^{1, 1}_{\mathrm{loc}} (\Omega)\) and for almost every \(x \in \Omega\), we have 
\begin{equation}
\label{eq_thaelohloh9UcahsieBei7Ah}
\int_0^{\diam (\Omega)} \brk*{\fint_{\Omega \cap B_r(x)} \abs{f - f (x)}}^p \frac{\dif r}{r^{1 + sp}}
 \le \frac{C}{1 - s} \brk[\big]{f^\sharp(x)}^{(1-s)p} \brk[\big]{ \varkappa (\Omega) 
 \maxfun{\abs{\Deriv f}} (x)}^{sp}.
\end{equation}
\end{proposition}

\Cref{proposition_local_GN_Wsp_W1p_BMO} is stronger than \cref{proposition_GN_Wsp_W1p_BMO} in the sense that the integration of the estimate \eqref{eq_thaelohloh9UcahsieBei7Ah} yields \eqref{eq_aijei2Xaighei4obeichiiKi}.

\Cref{proposition_local_GN_Wsp_W1p_BMO} is a counterpart of the interpolation involving maximal and sharp maximal function of derivatives \citelist{\cite{Lokharu_2011}*{(4)}}, which generalised a priori estimates in terms of maximal functions \citelist{\cite{Mazya_Shaposhnikova_1999}*{theorem 1}\cite{Kalamajska_1994}}; \cref{proposition_local_GN_Wsp_W1p_BMO} generalises the corresponding result for integer-order Sobolev spaces \citelist{\cite{Miyazaki_2020}*{remark 2.2}}.

\section{Interpolation between first-order Sobolev semi-norm and mean oscillation}

We explain how the tools of the previous section can be used to prove the fractional BMO Gagliardo--Nirenberg interpolation inequality as persented by Brezis and Mironescu's \citelist{\cite{Brezis_Mironescu_2021}*{lemma 15.7}}.

\begin{theorem}
\label{proposition_GN_Wsp_Wsp_BMO}
For every \(d \in \Nset \setminus \set{0}\), every \(s, s_1 \in \intvo{0}{1}\) and every \(p, p_1 \in (1, +\infty)\) 
satisfying \(s< s_1\) and \(s_1 p_1 = sp\), there exists a constant \(C > 0\) such that for every open convex set \(\Omega \subseteq \Rset^d\) satisfying \(\varkappa (\Omega)<\infty\)  and for every function \(f \in \dot{W}^{s_1, p_1} (\Omega) \cap\BMO(\Omega)\), one has
\(f \in \dot{W}^{s, p} (\Omega)\) and 
  \begin{equation}
  \label{eq_aeshohRe7chuFohyiizieyo4}
  \iint\limits_{\Omega \times \Omega}
  \frac{\abs{f (y) - f (x)}^p}{\abs{y - x}^{d + sp}}
  \dif y
  \dif x
  \le C 
  \norm{f}_{\mathrm{BMO}(\Omega)}^{p - p_1}\varkappa (\Omega)^{p_1}
  \iint\limits_{\Omega \times \Omega}
  \frac{\abs{f (y) - f (x)}^{p_1}}{\abs{y - x}^{d + s_1 p_1}}
  \dif y
  \dif x.
  \end{equation}
\end{theorem}

The proof of \cref{proposition_GN_Wsp_Wsp_BMO} will follow essentially the proof of \cref{proposition_GN_Wsp_W1p_BMO}, the main difference being the replacement of \cref{lemma_osc_MDu} by its easier fractional counterpart.

\begin{lemma}
\label{lemma_osc_Wsp}
For every \(p \in \intvo{1}{\infty}\), there exists a constant \(C > 0\) such that if the set \(\Omega \subseteq \Rset^d\) is open and convex, if \(s \in \intvo{0}{1}\) and
if \(f : \Omega \to \Rset\) is measurable, then for every \(r \in \intvo{0}{\diam(\Omega)}\) and every \(x \in \Omega\), 
\begin{equation}
\label{eq_jaC2Aedee8Pheefashiphuiw}
 \fint_{\Omega \cap B_r(x)} \abs{f - f (x)} 
 \le C \varkappa (\Omega) \, r^s \brk*{\int_{\Omega} \frac{\abs{f (y) - f (x)}^p}{\abs{y - x}^{d + sp}} \dif y}^\frac{1}{p}.
\end{equation}
\end{lemma}
\begin{proof}
By Hölder's inequality we have for every \(r \in \intvo{0}{\diam(\Omega)}\) and for every \(x \in \Omega\), 
\begin{equation}
 \int_{\Omega \cap B_r(x)} \abs{f - f (x)} 
 \le \brk*{\int_{\Omega} \frac{\abs{f (y) - f (x)}^p}{\abs{y - x}^{d + sp}} \dif y}^\frac{1}{p} \brk*{\int_{B_r (x)} \abs{y - x}^{\frac{d + sp}{p - 1}} \dif y}^{1 - \frac{1}{p}}.
\end{equation}
Noting that 
\begin{equation}
\int_{B_r (x)} \abs{y - x}^\frac{d + sp}{p - 1} \dif y
 = \C \frac{p - 1}{d + s p} \brk[\big]{r^s \mathcal{L}^d (B_r (x)) }^\frac{p}{p - 1}
 \le \C  \brk[\big]{r^s \mathcal{L}^d (B_r (x)) }^\frac{p}{p - 1},
\end{equation}
we reach the conclusion \eqref{eq_jaC2Aedee8Pheefashiphuiw} thanks to the definition of the geometric quantity \(\varkappa (\Omega)\) in \eqref{eq_waiC3theiwooDaithee1aeLa}.
\end{proof}

\begin{proof}[Proof of \cref{proposition_GN_Wsp_Wsp_BMO}]
We begin as in the proof of \cref{proposition_GN_Wsp_W1p_BMO}. Instead of \eqref{eq_pheiph4dei2uc0iefo5KaR0j}, we have by \cref{lemma_osc_Wsp},
\begin{equation}
\label{eq_ohThaequi9eiThoh0aidoh4a}
\begin{split}
 \int_0^\varrho \brk*{\fint_{\Omega \cap B_r(x)} \abs{f - f (x)}}^p &\frac{\dif r}{r^{1 + sp}}\\
 &\le \C \varkappa(\Omega)^p
 \brk*{\int_{\Omega} \frac{\abs{f (y) - f (x)}^{p_1}}{\abs{y - x}^{d + s_1p_1}} \dif y}^\frac{p}{p_1}
 \int_0^\varrho r^{(s_1 -s) p - 1} \dif r\\
 & \le \frac{\C \varkappa(\Omega)^p \varrho^{(s_1 - s)p}}{(s_1 -s) p} \brk*{\int_{\Omega} \frac{\abs{f (y) - f (x)}^{p_1}}{\abs{y - x}^{d + s_1p_1}} \dif y}^\frac{p}{p_1}.
 \end{split}
\end{equation} 
Next instead of \eqref{eq_Tah3Loosiela3aQuoothio2p}, we have 
\begin{equation}
\label{eq_ahYeipe1quaiZ3lei8voh4as}
\begin{split}
 \int_\varrho^{\diam (\Omega)} &\brk*{\fint_{\Omega \cap B_r(x)} \abs{f - f (x)}}^p \frac{\dif r}{r^{1 + sp}}\\
 &\le \C \brk*{\frac{\varkappa(\Omega)^p \varrho^{(s_1 - s)p} }{sp} \brk*{\int_{\Omega} \frac{\abs{f (y) - f (x)}^{p_1}}{\abs{y - x}^{d + s_1p_1}} \dif y}^\frac{p}{p_1}  + \frac{\norm{f}_{\BMO(\Omega)}^p}{(sp)^{p + 1} \varrho^{sp}}}.
 \end{split}
\end{equation}
Taking \(\varrho \in \intvo{0}{\diam (\Omega)}\) such that  
\begin{equation}
\label{eq_Iel5aigheeng3geeghiqu8ge}
 \norm{f}_{\mathrm{BMO}(\Omega)}^p = \varrho^{s_1 p}\varkappa(\Omega)^p \brk*{\int_{\Omega} \frac{\abs{f (y) - f (x)}^{p_1}}{\abs{y - x}^{d + s_1p_1}} \dif y}^\frac{p}{p_1}
\end{equation}
if possible, and otherwise taking \(\varrho \defeq \diam (\Omega)\),
we obtain, since \(s_1 p_1 = sp\), by \eqref{eq_ohThaequi9eiThoh0aidoh4a}, \eqref{eq_ahYeipe1quaiZ3lei8voh4as} and \eqref{eq_Iel5aigheeng3geeghiqu8ge}
\begin{equation}
\label{eq_ohng5eeRas3Ohcithoociit4}
\begin{split}
 \int_0^{\diam \Omega} \brk*{\fint_{\Omega \cap B_r(x)} \abs{f - f (x)}}^p &\frac{\dif r}{r^{1 + sp}}\\[-1ex]
 &\le \C \norm{f}_{\BMO(\Omega)}^{p - p_1} \varkappa(\Omega)^{p_1} 
 \int_{\Omega} \frac{\abs{f (y) - f (x)}^{p_1}}{\abs{y - x}^{d + s_1p_1}} \dif y.
 \end{split}
\end{equation}
We conclude by integration of \eqref{eq_ohng5eeRas3Ohcithoociit4}.
\end{proof}

As previously, we point out that the estimate \eqref{eq_ohng5eeRas3Ohcithoociit4} admits a localised version, which is the fractional counterpart of \cref{proposition_local_GN_Wsp_W1p_BMO}.

\begin{proposition}
For every \(d \in \Nset \setminus \set{0}\), every \(s, s_1 \in \intvo{0}{1}\) and every \(p, p_1 \in (1, +\infty)\) 
satisfying \(s< s_1\) and \(s_1 p_1 = sp\), there exists a constant \(C > 0\) such that for every open convex set \(\Omega \subseteq \Rset^d\) satisfying \(\varkappa (\Omega)<\infty\), for every measurable function \(f : \Omega \to \Rset\) and for every \(x \in \Omega\),
\begin{equation}
\label{eq_ael6ohf6rooK4shaiMe0thai}
\int_0^{\diam (\Omega)} \brk*{\fint_{\Omega \cap B_r(x)} \abs{f - f (x)}}^p \frac{\dif r}{r^{1 + sp}}\le C  \brk[\big]{f^\sharp(x)}^{p - p_1}\varkappa (\Omega)^{p_1}
 \int_{\Omega} \frac{\abs{f (y) - f (x)}^{p_1}}{\abs{y - x}^{d + s_1p_1}} \dif y.
\end{equation}
\end{proposition}

The estimate \eqref{eq_aeshohRe7chuFohyiizieyo4} can be seen as a consequence of the integration of \eqref{eq_ael6ohf6rooK4shaiMe0thai}.

\section{Higher-order fractional spaces estimates}

The last ingredient to obtain the full scale of Gagliardo--Nirenberg interpolation inequalities between fractional Sobolev--Slobodeckiĭ spaces and the bounded mean oscillation space is the following estimate.

\begin{theorem}
\label{proposition_higher_order}
For every \(d \in \Nset \setminus \set{0}\), every \(k_1 \in \Nset \setminus\set{0}\), every \(\sigma_1 \in \intvo{0}{1}\) and every \(p, p_1 \in \intvo{1}{\infty}\) satisfying
\begin{equation}
\label{eq_Yoxaikeengah4deil3Aingae}
k_1p = (k_1 + \sigma_1)p_1,
\end{equation}
there exists a constant \(C > 0\) such that for every function \(f \in \dot{W}^{k_1 + \sigma_1, p_1} (\Rset^d)\cap \mathrm{BMO} (\Rset^d)\), one has \(f \in \dot{W}^{k_1, p} (\Rset^d)\) and 
\begin{equation}
\label{eq_ohjooYai3xeijo1aiphiivie} 
\int_{\Rset^d} \abs{D^{k_1} f}^p
\le C \norm{f}_{\mathrm{BMO} (\Rset^d)}^{p - p_1} 
\smashoperator{\iint_{\Rset^d \times \Rset^d}} \frac{\abs{D^{k_1} f(x) - D^{k_1} f(y)}^{p_1}}{\abs{x - y}^{d + \sigma_1p_1}} \dif x \dif y.
\end{equation}
\end{theorem}

As a consequence of \cref{proposition_higher_order}, we have that \(f \in \dot{W}^{k + \sigma, p}(\Rset^d)\) whenever \(k \in \Nset\), \(\sigma \in \intvr{0}{1}\) and \(p \in \intvo{1}{\infty}\) satisfy \(k + \sigma < k_1 + \sigma_1\) and 
\((k + \sigma)p = (k_1 + \sigma_1)p_1\).
Indeed for \(\sigma = 0\) and \(k = k_1\), this follows from \cref{proposition_higher_order} and then for \(k \in \set{1, \dotsc, k_1 - 1}\) by the Gagliardo--Nirenberg interpolation inequality for integer-order Sobolev space \citelist{\cite{Strzelecki_2006}\cite{Miyazaki_2020}}; 
for \(0 < \sigma < 1\) and \(k = 0\) one then uses \cref{proposition_GN_Wsp_W1p_BMO} whereas for \(0 < \sigma < 1\) and \(k \in \Nset \setminus \set{0}\) one uses the classical fractional Gagliardo--Nirenberg interpolation inequality \cite{Brezis_Mironescu_2018}.

\begin{proof}[Proof of \cref{proposition_higher_order}]
Fixing a function \(\eta \in C^\infty_c (\Rset^d)\) such that \(\int_{\Rset^d} \eta = 1\) and \(\supp \eta \subseteq B_1\), we have for every \(x \in \Rset^d\) and every \(\varrho \in \intvo{0}{\infty}\),
\begin{equation}
\label{eq_iemai1quuyeiSh2eebo9aiw8}
 D^{k_1} f (x)
= \frac{1}{\varrho^d} \int_{\Rset^d} \eta\brk*{\tfrac{x - y}{\varrho}} \brk[\big]{D^{k_1} f (x) - D^{k_1} f (y)} \dif y + \frac{1}{\varrho^d} \int_{\Rset^d} \eta\brk*{\tfrac{x - y}{\varrho}} D^{k_1} f (y) \dif y .
\end{equation}
We estimate the first term in the right-hand side of \eqref{eq_iemai1quuyeiSh2eebo9aiw8} by Hölder's inequality
\begin{equation}
\label{eq_iumooxoh6ohhePijahw6leev}
\begin{split}
 \abs[\bigg]{\frac{1}{\varrho^d} \int_{\Rset^d}& \eta\brk*{\tfrac{x - y}{\varrho}} \brk[\big]{D^{k_1} f (x) - D^{k_1} f (y)} \dif y}\\
 &\le \frac{\C}{\varrho^d} \brk*{\int_{\Rset^d} \frac{\abs{D^{k_1} f (x) - D^{k_1} f (y)}^{p_1}}{\abs{x - y}^{d+\sigma_1 p_1}}\dif x}^\frac{1}{p_1}
 \brk*{\int_{B_\varrho(x)} \abs{x - y}^{\frac{d+\sigma_1 p_1}{p_1 - 1}}\dif x}^{1 - \frac{1}{p_1}}\\
 &\le \C \varrho^{\sigma_1} \brk*{\int_{\Rset^d} \frac{\abs{D^{k_1} f (x) - D^{k_1} f (y)}^{p_1}}{\abs{x - y}^{d+\sigma_1 p_1}}\dif x}^\frac{1}{p_1}.
\end{split}
\end{equation}
For the second-term in the right-hand side of \eqref{eq_iemai1quuyeiSh2eebo9aiw8}, for every \(x \in \Rset^d\), we have by weak differentiability, 
\begin{equation}
\label{eq_Shai9yauPh6Ohph3phei2Uic}
\begin{split}
\frac{1}{\varrho^d} \int_{\Rset^d} \eta\brk*{\tfrac{x - y}{\varrho}} D^{k_1} f (y) \dif y
 &= \frac{1}{\varrho^{d + k_1}} \int_{\Rset^d} D^{k_1} \eta\brk*{\tfrac{x - y}{\varrho}} f (y) \dif y\\
 &= \frac{1}{\varrho^{2 d + k_1}} \int_{\Rset^d} \int_{\Rset^d} D^{k_1} \eta\brk*{\tfrac{x - y}{\varrho}} \eta\brk*{\tfrac{x - z}{\varrho}} \brk[\big]{f (y) - f (z)} \dif y \dif z,
\end{split}
\end{equation}
and thus by \eqref{eq_Shai9yauPh6Ohph3phei2Uic} and by definition of bounded mean oscillation \eqref{eq_oothaKahp8pe7OGheix4phoo}, we have
\begin{equation}
\label{eq_egai1Isae4vie5opeida5le8}
 \abs*{\frac{1}{\varrho^d} \int_{\Rset^d} \eta\brk*{\tfrac{x - y}{\varrho}} D^{k_1} f (y) \dif y}
 \le \frac{\C}{\varrho^{k_1}} \norm{f}_{\mathrm{BMO}(\Rset^d)}. 
\end{equation}
Choosing \(\varrho \in \intvo{0}{\infty}\) such that 
\begin{equation}
 \varrho^{k_1 + \sigma_1} \brk*{\int_{\Rset^d} \frac{\abs{D^{k_1} f (x) - D^{k_1} f (y)}^{p_1}}{\abs{x - y}^{d+\sigma_1 p_1}}\dif y}^\frac{1}{p_1} = \norm{f}_{\mathrm{BMO}(\Rset^d)},
\end{equation}
we get from \eqref{eq_iemai1quuyeiSh2eebo9aiw8}, \eqref{eq_iumooxoh6ohhePijahw6leev} and \eqref{eq_egai1Isae4vie5opeida5le8}, for every \(x \in \Rset^d\),
\begin{equation}
 \abs{D^{k_1} f (x)}
 \le \C \norm{f}_{\mathrm{BMO}(\Rset^d)}^{1 - \frac{k_1}{k_1 + \sigma_1}} \brk*{\int_{\Rset^d} \frac{\abs{D^{k_1} f (x) - D^{k_1} f (y)}^{p_1}}{\abs{x - y}^{d+\sigma_1 p_1}}\dif y}^\frac{k_1}{(k_1 + \sigma_1) p_1}
 ,
\end{equation}
and thus in view of the condition \eqref{eq_Yoxaikeengah4deil3Aingae}, the estimate \eqref{eq_ohjooYai3xeijo1aiphiivie} follows by integration.
\end{proof}

\Cref{proposition_higher_order} also admits a localised version involving the sharp maximal function which follows from the replacement of \(\norm{f}_{\mathrm{BMO}(\Rset^d)}\) by \(f^\sharp (x)\) in \eqref{eq_egai1Isae4vie5opeida5le8}.

\begin{proposition}
\label{proposition_higher_order_local}
For every \(d \in \Nset \setminus \set{0}\), every \(k_1 \in \Nset \setminus\set{0}\), every \(\sigma_1 \in \intvo{0}{1}\) and every \(p_1 \in \intvo{1}{\infty}\),
there exists a constant \(C>0\) such that for every function \(f \in \dot{W}^{k_1, 1}_{\mathrm{loc}} (\Rset^d)\) and every \(x \in \Rset^d\), 
\begin{equation}
\label{eq_jaeliu1aeph2tooTh8Wul5ox}
\abs{D^{k_1} f (x)}
 \le C \brk[\big]{f^\sharp (x)}^{1 - \frac{k_1}{k_1 + \sigma_1}} \brk*{\int_{\Rset^d} \frac{\abs{D^{k_1} f (x) - D^{k_1} f (y)}^{p_1}}{\abs{x - y}^{d+\sigma_1 p_1}}\dif y}^\frac{k_1}{(k_1 + \sigma_1) p_1}.
\end{equation}
\end{proposition}

As previously, the integration of \eqref{eq_jaeliu1aeph2tooTh8Wul5ox} yields \eqref{eq_ohjooYai3xeijo1aiphiivie}.

\end{document}